\documentclass[12pt]{article}
\title{Induced universal graphs for families of small graphs}
\author{
        James Trimble \\
                School of Computing Science\\
        University of Glasgow, Glasgow, Scotland \\
        \href{mailto:j.trimble.1@research.gla.ac.uk}{j.trimble.1@research.gla.ac.uk}
}
\date{\today}

\usepackage[letterpaper, total={6in, 8in}]{geometry}
\usepackage[x11names, svgnames, rgb]{xcolor}
\usepackage[utf8]{inputenc}
\usepackage{booktabs}
\usepackage{hyperref}
\usepackage{tikz}
\usepackage{amsthm}
\usetikzlibrary{decorations, decorations.pathreplacing, decorations.pathmorphing,
    calc, backgrounds, positioning, tikzmark, patterns, matrix, fit, graphs,
    snakes, arrows, shapes, calligraphy}
\usepackage{amsmath}
\usepackage{subfigure}
\usepackage{cleveref}
\usepackage[ruled,vlined]{algorithm2e}
\usepackage{algpseudocode}

\usepackage{caption}
\crefname{algorithm}{Algorithm}{Algorithms}
\Crefname{algorithm}{Algorithm}{Algorithms}
\crefname{algocf}{Algorithm}{Algorithms}
\Crefname{algocf}{Algorithm}{Algorithms}
\crefname{figure}{Figure}{Figures}
\Crefname{figure}{Figure}{Figures}
\crefname{table}{Table}{Tables}
\Crefname{table}{Table}{Tables}
\crefname{section}{Section}{Sections}
\Crefname{section}{Section}{Sections}

\newcommand{\calF}{\ensuremath{\mathcal{F}}}
\newcommand{\calG}{\ensuremath{\mathcal{G}}}
\newcommand{\AlgVar}[1]{\mathit{#1}}

\newcommand{\lineref}[1]{line~\ref{#1}}

\newcommand{\linerangeref}[2]{lines~\ref{#1} to~\ref{#2}}
\newcommand{\Lineref}[1]{Line~\ref{#1}}

\algnewcommand{\LeftComment}[1]{\(\triangleright\) #1}

\newtheorem{proposition}{Proposition}

\definecolor{uofguniversityblue}{rgb}{0, 0.219608, 0.396078}

\definecolor{uofgheather}{rgb}{0.356863, 0.32549, 0.490196}
\definecolor{uofgaquamarine}{rgb}{0.603922, 0.72549, 0.678431}
\definecolor{uofgslate}{rgb}{0.309804, 0.34902, 0.380392}
\definecolor{uofgrose}{rgb}{0.823529, 0.470588, 0.709804}
\definecolor{uofgmocha}{rgb}{0.709804, 0.564706, 0.47451}

\definecolor{uofglawn}{rgb}{0.517647, 0.741176, 0}
\definecolor{uofgcobalt}{rgb}{0, 0.615686, 0.92549}
\definecolor{uofgturquoise}{rgb}{0, 0.709804, 0.819608}
\definecolor{uofgsunshine}{rgb}{1.0, 0.862745, 0.211765}
\definecolor{uofgpumpkin}{rgb}{1.0, 0.72549, 0.282353}
\definecolor{uofgthistle}{rgb}{0.584314, 0.070588, 0.447059}
\definecolor{uofgpillarbox}{rgb}{0.701961, 0.047059, 0}
\definecolor{uofglavendar}{rgb}{0.356863, 0.301961, 0.580392}

\definecolor{uofgsandstone}{rgb}{0.321569, 0.278431, 0.231373}
\definecolor{uofgforest}{rgb}{0, 0.317647, 0.2}
\definecolor{uofgburgundy}{rgb}{0.490196, 0.133333, 0.223529}
\definecolor{uofgrust}{rgb}{0.603922, 0.227451, 0.023529}

\begin{document}

\maketitle
\begin{abstract}
We present exact and heuristic algorithms that find, for a given family
of graphs, a graph that contains each member of the family as an induced
subgraph.
For $0 \leq k \leq 6$, we give the minimum number of vertices $f(k)$ in a graph containing all
$k$-vertex graphs as induced subgraphs, and show
that $16 \leq f(7) \leq 18$.  For $0 \leq k \leq 5$, we also give the counts of
such graphs, as generated by brute-force computer search.  We give additional
results for small graphs containing all trees on $k$ vertices.
\end{abstract}

\section{Introduction}

Given a collection $\calF$ of graphs, graph $G$ is \emph{induced universal} for
$\calF$ if and only if each element of $\calF$ is an induced subgraph of $G$.
Graph $G$ is a \emph{minimal} induced universal graph for $\calF$ if it has as
few vertices as possible.  The problem of finding a minimal induced universal graph
of a family of graphs generalises the minimum common supergraph problem
\cite{DBLP:journals/computing/BunkeJK00}, which applies only to families
containing exactly two graphs.

We write $\calF(k)$ to denote the family of all graphs on $k$ vertices,
and we write $f(k)$ to denote the order (that is, the
number of vertices) of a minimal induced universal graph for $\calF(k)$.
We write $F(k)$ to denote the number of non-isomorphic graphs of order $f(k)$
that are induced universal for $\calF(k)$.
To give an example, we have $f(3)=5$ and $F(3)=5$; all five of the minimal
induced universal graphs
for $\calF(3)$ are shown in \Cref{fig:univ3a}. \Cref{fig:univ3b} shows the four
graphs in $\calF(3)$ as induced subgraphs of a single induced universal graph.

Moon showed that $f(k) \leq 2^{(k-1)/2}$
\cite{moon_1965}, and Alon showed that $f(k) = (1 + o(1))2^{(k-1)/2}$
\cite{alon2017asymptotically}.  There is an extensive literature on bounds on
the order of minimal induced universal subgraphs for many families of graphs;
see the references in \cite{alon2017asymptotically}.
However, to my knowledge the only
existing systematic attempt to find \emph{exact} values for families of small
graphs is an answer by James Preen on the Mathematics Stack Exchange website that
presents results for families of small connected graphs obtained by brute-force
search with the Maple programming language \cite{preen_math_se}.

\begin{figure}[htb]
    \centering

    \subfigure[][The five induced universal graphs of order 5 for \calF(3) (that is, for the family of all 3-vertex graphs)] {
      \begin{tikzpicture}[>=latex',line join=bevel,scale=.4]
        \graph [nodes={draw, circle, minimum width=.49cm, inner sep=1pt}, circular placement, radius=0.78cm,
                clockwise=5] {
                    0,1,2,3,4;
           3 -- 0;
           4 -- 0;
           4 -- 1;
           4 -- 3;
        };
      \end{tikzpicture}
      \qquad
      \begin{tikzpicture}[>=latex',line join=bevel,scale=.4]
        \graph [nodes={draw, circle, minimum width=.49cm, inner sep=1pt}, circular placement, radius=0.78cm,
                clockwise=5] {
                    0,1,2,3,4;
          3 -- 0;
          4 -- 0;
          4 -- 1;
          4 -- 2;
          4 -- 3;
        };
      \end{tikzpicture}
      \qquad
      \begin{tikzpicture}[>=latex',line join=bevel,scale=.4]
        \graph [nodes={draw, circle, minimum width=.49cm, inner sep=1pt}, circular placement, radius=0.78cm,
                clockwise=5] {
                    0,1,2,3,4;
          3 -- 0;
          3 -- 1;
          4 -- 0;
          4 -- 2;
          4 -- 3;
        };
      \end{tikzpicture}
      \qquad
      \begin{tikzpicture}[>=latex',line join=bevel,scale=.4]
        \graph [nodes={draw, circle, minimum width=.49cm, inner sep=1pt}, circular placement, radius=0.78cm,
                clockwise=5] {
                    0,1,2,3,4;
          3 -- 0;
          3 -- 1;
          4 -- 0;
          4 -- 1;
          4 -- 3;
        };
      \end{tikzpicture}
      \qquad
      \begin{tikzpicture}[>=latex',line join=bevel,scale=.4]
        \graph [nodes={draw, circle, minimum width=.49cm, inner sep=1pt}, circular placement, radius=0.78cm,
                clockwise=5] {
                    0,1,2,3,4;
          3 -- 0;
          3 -- 1;
          4 -- 0;
          4 -- 1;
          4 -- 2;
          4 -- 3;
        };
      \end{tikzpicture}
      \label{fig:univ3a}
    }

    \par\bigskip
    \subfigure[][A demonstration that the first graph in \Cref{fig:univ3a} is induced universal
            for $\calF(3)$.  For each graph $G$ in $\calF(3)$ ($I_3$, $K_3$, $P_3$, and
            a graph with a single edge), an induced subgraph isomorphic to $G$ is shown in a
            single color.] {
      \begin{tikzpicture}[>=latex',line join=bevel,scale=.4]
        \graph [nodes={draw, circle, minimum width=.49cm, inner sep=1pt}, circular placement, radius=0.78cm,
                clockwise=5] {
                    0,1[fill=uofgpumpkin],2[fill=uofgpumpkin],3[fill=uofgpumpkin],4;
           3 -- 0;
           4 -- 0;
           4 -- 1;
           4 -- 3;
        };
      \end{tikzpicture}
      \qquad
      \begin{tikzpicture}[>=latex',line join=bevel,scale=.4]
        \graph [nodes={draw, circle, minimum width=.49cm, inner sep=1pt}, circular placement, radius=0.78cm,
                clockwise=5] {
                    0[fill=uofgturquoise],1,2,3[fill=uofgturquoise],4[fill=uofgturquoise];
           4 -- 1;
           3 --[uofgturquoise, ultra thick] 0;
           4 --[uofgturquoise, ultra thick] 0;
           4 --[uofgturquoise, ultra thick] 3;
        };
      \end{tikzpicture}
      \qquad
      \begin{tikzpicture}[>=latex',line join=bevel,scale=.4]
        \graph [nodes={draw, circle, minimum width=.49cm, inner sep=1pt}, circular placement, radius=0.78cm,
                clockwise=5] {
                    0,1[fill=uofgrose],2,3[fill=uofgrose],4[fill=uofgrose];
           3 -- 0;
           4 -- 0;
           4 --[uofgrose, ultra thick] 1;
           4 --[uofgrose, ultra thick] 3;
        };
      \end{tikzpicture}
      \qquad
      \begin{tikzpicture}[>=latex',line join=bevel,scale=.4]
        \graph [nodes={draw, circle, minimum width=.49cm, inner sep=1pt}, circular placement, radius=0.78cm,
                clockwise=5] {
                    0,1[fill=uofgmocha],2[fill=uofgmocha],3,4[fill=uofgmocha];
           3 -- 0;
           4 -- 0;
           4 --[uofgmocha, ultra thick] 1;
           4 -- 3;
        };
      \end{tikzpicture}
      \label{fig:univ3b}
    }
\caption{Minimal induced universal graphs for the family of all graphs on 3 vertices}
\label{fig:univ3}
\end{figure}
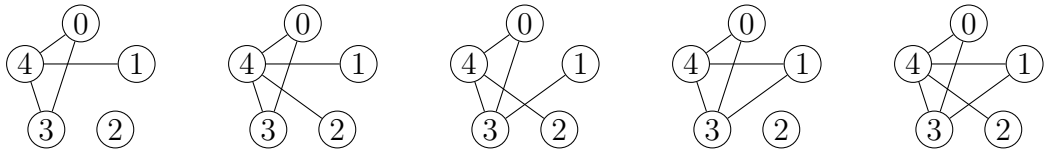
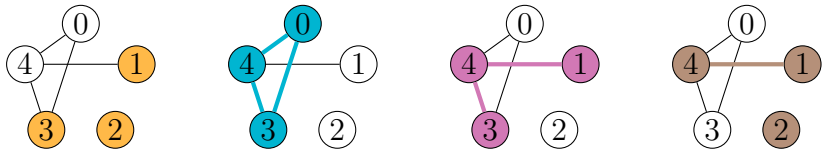

This paper uses a brute-force approach similar to that of Preen to find minimal
induced universal graphs for families of small graphs.  Our contributions
are (1) an ordering strategy to reduce the number of subgraph isomorphism
calls required by the brute force algorithm; (2) exact values for $f(k)$ $(5 \leq k \leq 6)$
and $F(k)$ $(1 \leq k \leq 5)$; (3) an upper bound of 18 for $f(7)$; (4) a hill-climbing
search algorithm to find small (but not necessarily optimal) induced universal graphs
(5) exact values of the order of a minimal induced universal graph
for the families of all trees on $k$ vertices, for $1 \leq k \leq 8$.

The sequel is structured as follows.
Section~\ref{sec:method} describes the basic brute force method.
Section~\ref{sec:iteration-order} compares four methods for sorting
the graphs in $\calF$, and compares their effect on the number of subgraph
isomorphism calls made by the brute-force algorithm.
Sections~\ref{sec:results5} presents exact
values of $f(k)$ and $F(k)$ for $k \leq 5$.
Section~\ref{sec:f6} gives the value of $f(6)$ and and
Section~\ref{sec:f7} gives bounds on $f(7)$; in each of these sections, the lower
bound is proven and a graph obtained by heuristic search demonstrating the
upper bound is given.  Finally, Section~\ref{sec:trees} describes a specialised
algorithm for families of trees, and gives results for
these families.

\section{Generating all induced universal graphs}\label{sec:method}

Given $\calF$ and $n$,
we use the brute-force approach
shown in \Cref{alg:brute-force} --- which is essentially the same as the method
described by Preen \cite{preen_math_se} ---
to find the set of
$n$-vertex graphs that are induced universal for the family $\calF$.
The entry point is the function $\FuncSty{AllInducedUniversalGraphs}()$.
\Lineref{TryEachG} tests in turn each candidate graph $G$ from the family
of $n$-vertex graphs, and adds to the collection $\calG$ those that are induced universal
for $\calF$.  The function $\FuncSty{IsInducedUniversal}()$ tests
whether a graph $G$ is induced universal for $\calF$ by checking for
each $H \in \calF$ that $H$ is isomorphic to an induced subgraph of $G$.

For the collection $\calF(n)$ of candidate graphs on \lineref{TryEachG}, our
implementation uses graphs generated using Brendan McKay's \texttt{geng}
program \cite{DBLP:journals/jal/McKay98}.\footnote{The program \texttt{geng} is distributed
with Nauty; we used version 27.r3.\url{https://pallini.di.uniroma1.it/}}
These are read in \texttt{graph6}
format\footnote{\url{https://users.cecs.anu.edu.au/~bdm/data/formats.html}}
from a text file as the program proceeds, and therefore only one candidate
graph at a time needs to be stored in memory.

Our implementation is written purely in Python and run using the CPython
interpreter.
On \lineref{CallSubIso}, the function $\FuncSty{InducedSubIso}()$
calls an algorithm for the induced
subgraph isomorphism decision problem; for this,
our program uses an implementation by the author of the McSplit
algorithm \cite{DBLP:conf/ijcai/McCreeshPT17}.
McSplit was designed for
the more general maximum common induced subgraph problem,
but it can be used for induced subgraph isomorphism with a simple
modification: we backtrack when the calculated upper bound
on the order of a common subgraph is less the order of the smaller
of the two input graphs.
This method for solving
the subgraph isomorphism problem is very fast if both input graphs
are small, as the graphs considered in this paper are.
While McSplit is suitable for our purposes in this paper,
we do not claim that it is the fastest induced subgraph isomorphism
solver for very small graphs; it would be useful future work to perform
an experimental comparison with other subgraph isomorphism solvers.

\begin{algorithm}[h!]
\DontPrintSemicolon
\nl $\FuncSty{IsInducedUniversal}(\calF,G)$ \;
\nl \KwData{A family of graphs $\calF$ and a graph $G$.}
\nl \KwResult{A boolean value indicating whether $G$ is induced universal for $\calF$.}
\nl \Begin{
\nl    \For {$H \in \calF$ \label{HLoop}}{
\nl      \lIf {$\neg\FuncSty{InducedSubIso}(H, G)$\label{CallSubIso}}{$\KwSty{return}$~$\AlgVar{false}$}
       }
\nl   $\KwSty{return}$~$\AlgVar{true}$ \;
    }
\medskip
\nl $\FuncSty{AllInducedUniversalGraphs}(\calF,n)$ \label{BruteForceFun} \;
\nl \KwData{A family of graphs $\calF$ and a natural number $n$.}
\nl \KwResult{The set of all graphs of order $n$ that are induced universal for $\calF$.}
\nl \Begin{
\nl    $\calG \gets \emptyset$ \;
\nl    \For {$G \in \calF(n)$}{
\nl      \lIf {$\FuncSty{IsInducedUniversal}(\calF,G)$\label{TryEachG}}{$\calG \gets \calG \cup \{G\}$}
       }
\nl   $\KwSty{return}$~$\calG$ \;
    }
\caption{A brute-force algorithm for finding all order-$n$ induced universal subgraphs of a family $\calF$ of graphs.
	The entry point is $\FuncSty{AllInducedUniversalGraphs}(\calF,n)$.}
\label{alg:brute-force}
\end{algorithm}

The full set of experiments described in this paper, run sequentially,
took less than a day to complete on a laptop with an Intel Core i5-6200U CPU
and 8 GB RAM.\footnote{The code and
results from this paper, including lists of minimal induced universal graphs
in graph6 format, are available from
\url{https://github.com/jamestrimble/small-universal-graphs}}

\section{Iteration order for $\calF$}\label{sec:iteration-order}

When determining on \lineref{HLoop} of \Cref{alg:brute-force} whether a graph $G$ contains
every element of $\calF$ as an induced subgraph, are some iteration orders of
$\calF$ better than others?  For a graph $G$ that contains a copy of each
element of $\calF$, the iteration order is of no importance;
every element of $\calF$ must be checked before $\AlgVar{true}$ is returned.
For a graph $G$ that
does \emph{not} contain a copy of each element of $\calF$, however, we would like
to iterate over $\calF$ in an order such that graphs that are likely to fail the
subgraph isomorphism test are checked early; that way, we can quickly return
$\AlgVar{false}$ and avoid unnecessary calls to the subgraph isomorphism algorithm.

We begin by considering the case where our family of graphs $\calF$
is $\calF(k)$ for some $k$; that is,
$\calF$ is the set of all graphs of order $k$.
As Diaconis and Chatterjee note, \cite{chatterjee2021isomorphisms}, the complete
graph $K_n$ is the `hardest' graph of order $n$ for a random graph to contain.
This led me
to consider two orderings in which $K_k$ and its complement $I_k$ come
first.  The first of these approaches is to sort $\calF(k)$ in descending
order of automorphism count, so that
graphs with large automorphism groups appear first in the list.  (Intuitively, if a graph
has few automorphisms then we can generate many different labelled graphs
by permuting the vertex labels.  This gives more `opportunities' for the
graph to be an induced subgraph of $G$.)
The second approach considered is to place graphs
with unusually high or low edge counts, as measured by the absolute value of
${2|E(G)| - {|V(G)| \choose 2}}$, first.

\Cref{fig:scatter} examines whether graphs that appear early in the orderings
of $\calF$
generated by each of these two strategies are contained in few graphs with
a given number of vertices, as we would hope to be the case.
For this experiment, we let $\calF = \calF(5)$, and we consider the problem
of finding induced universal graphs of order 8.
Each point on the plots
represents one of the 34 graphs of order 5 (fewer than 34 points are visible
due to exactly-overlapping points, as we would expect).  On the vertical
axis, we have the number of graphs in $\calF(8)$ containing $G$; we would like
graphs for which this value is small to appear early in our order.
On the horizontal axes we have each of our two measures: size of automorphism group
in the first plot and `extremeness' of edge count in the second plot.
As we can see from the first plot, the automorphism measure has a strong correlation
with the number of graphs containing $G$ as an induced subgraph, suggesting that
sorting by this measure could be a good strategy for reducing the number of
subgraph isomorphism calls in \Cref{alg:brute-force}.

\begin{figure}[htb]
    \centering

      \includegraphics*{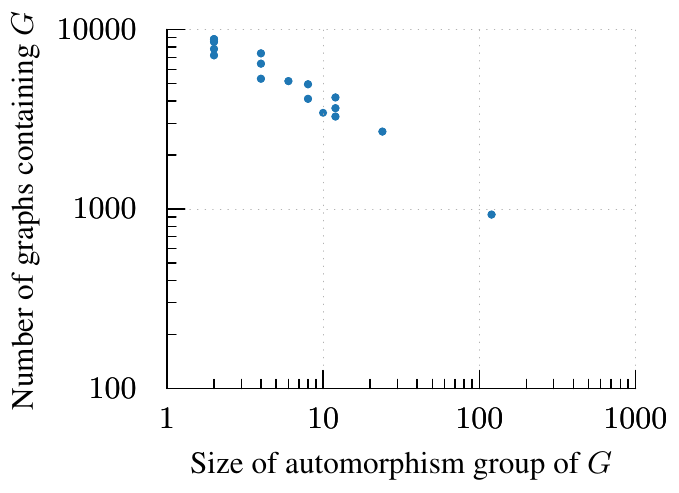}
     \qquad
      \includegraphics*{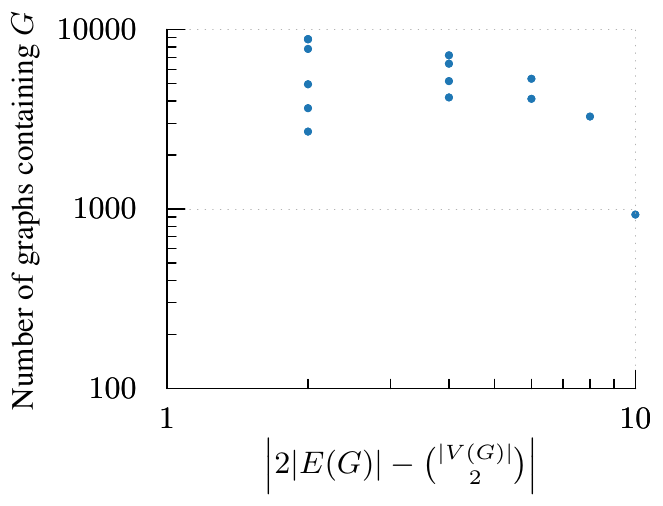}
    \caption{For each graph $G \in \calF(5)$, we plot on the vertical axis
    the number of graphs in $\calF(8)$ that contain $G$ as an induced subgraph.
    This value has a strong negative correlation with the size of the automorphism
    group of $G$ (left plot), but a weaker correlation with a measure of how
    extreme the edge count of $G$ is (right plot).
    All axis scales are logarithmic.}
\label{fig:scatter}
\end{figure}

\subsection{Testing the strategies: $\calF = \calF{5}$}

We now test whether each of our ordering strategies results, as hoped,
in a reduced number of calls to the subgraph isomorphism function
in \Cref{alg:brute-force}.  As our first test case, we consider 
$\FuncSty{AllInducedUniversalGraphs}(\calF(5),9)$ ---
finding all induced universal graphs of order 9 for the family of
all graphs of order 5.  We consider four strategies.

The first two of these are the automorphism-count and `edge-count extremeness'
measures described above.  The fourth strategy is a random order.
Finally, we used one additional strategy that we call ``almost random''.
Under this strategy, $I_5$ and $K_5$ are placed at the start of the list,
and the remaining 32 graphs are in random order.  This strategy was not
intended to be useful in practice, but to give insight into whether
the success of the first two strategies was purely due to having
$I_5$ and $K_5$ at the start of the list of graphs. \Cref{tab:ordering-strategies}
summarises these ordering strategies.

\begin{table}[h!]
\centering
 \begin{tabular}{p{0.25\linewidth} p{0.55\linewidth}}
 \toprule
    Name & Description of strategy \\ [0.5ex]
 \midrule
    Automorphisms & Sort in descending order of automorphism group size \\
    Edges & Sort in descending order of $\big|{2|E(G)| - {|V(G)| \choose 2}}\big|$ 
            (thus, give priority to graphs that are very sparse or very dense)\\
    Almost-random & Shuffle the list of graphs in $\calF$, then move $K_k$ and $I_k$ to the start if they are present \\
    Random & Shuffle the list of graphs in $\calF$ \\
 \bottomrule
\end{tabular}
\caption{Ordering strategies tested for the brute-force algorithm}
\label{tab:ordering-strategies}
\end{table}

For each of these strategies, we ran the program 50 times.

The results are shown in \Cref{fig:second-experiment}.  Over 50 runs,
the automorphisms strategy required a mean of 304,746 calls to the subgraph
isomorphism solver, while the edges strategy required marginally fewer
calls:  a mean of 304,701.  Surprisingly, the almost-random
strategy had the lowest average number of calls of all the strategies,
at 304,387.  The random strategy fared poorly, requiring over 800,000
calls on average.
The variance in number of calls was small for all but the random strategy,
with between 304,000 and 305,000 calls being required for each run.

\begin{figure}[h!]
    \centering

    \includegraphics*[width=0.7\textwidth]{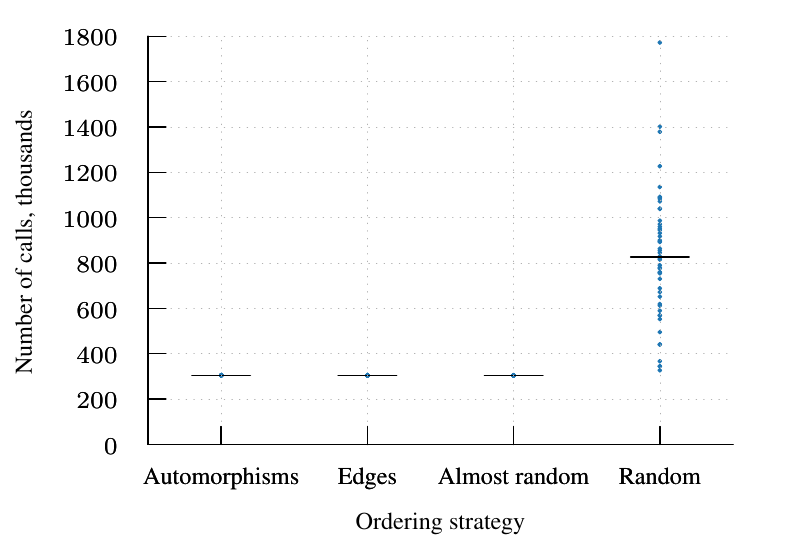}

    \caption{Using each of our four order strategies, we ran
        $\FuncSty{AllInducedUniversalGraphs}(\calF(5),9)$ 50 times. The plot
        shows the number calls to the subgraph isomorphism algorithm, in thousands.
        Horizontal lines show average (mean) number of calls.}
\label{fig:second-experiment}
\end{figure}

We note, first, that edges strategy appears to be at least as effective
as the automorphisms strategy for this instance, despite the lower
correlation observed in \Cref{fig:scatter}.  Perhaps most surprising is
that the almost-random strategy performed better than either of these two
strategies; our expectation had been that this simple strategy would
give some but not all of the benefit of the first two strategies.
We conjecture that the almost-random strategy is effective because it
quickly checks a diverse range of graphs.  Consider, for example,
\Cref{tab:graphs-by-autom-count}, which shows the graphs of order 5 listed
in descending order of automorphism count.  This is the order in which they
are checked under the automorphisms strategy (with the order within each
row chosen randomly).  Note that the two graphs in the second row are
similar to the graphs in the first row; the first graphs in each of the first
two rows, for example, have a maximum common induced subgraph of size 4.
It seems likely that if a graph contains both of the graphs in the first row of
the table as induced subgraphs, then it may also contain the graphs in the
second row. The almost-random strategy is likely to choose graphs
that are quite different from a clique and an independent set after the first
two graphs, perhaps increasing the probability that one of these graphs
will not be contained in the larger graph.

\begin{table}[h!]
\centering
\begin{tabular}{c l}
 \toprule
 Automorphisms & Graphs \\ [0.5ex]
 \midrule
120 & \begin{tikzpicture}[>=latex',line join=bevel,scale=.3]
    {
         \tikzstyle{every node}=[draw, circle, inner sep=1pt, minimum size=.1cm]
         \node (0) at (6.123233995736766e-17,1.0) {};\node (1) at (-0.9510565162951535,0.3090169943749475) {};\node (2) at (-0.5877852522924732,-0.8090169943749473) {};\node (3) at (0.5877852522924729,-0.8090169943749476) {};\node (4) at (0.9510565162951536,0.3090169943749472) {};
         \draw (0) -- (1); \draw (0) -- (2); \draw (0) -- (3); \draw (0) -- (4); \draw (1) -- (2); \draw (1) -- (3); \draw (1) -- (4); \draw (2) -- (3); \draw (2) -- (4); \draw (3) -- (4);
    };
  \end{tikzpicture}\quad\begin{tikzpicture}[>=latex',line join=bevel,scale=.3]
    {
         \tikzstyle{every node}=[draw, circle, inner sep=1pt, minimum size=.1cm]
         \node (0) at (6.123233995736766e-17,1.0) {};\node (1) at (-0.9510565162951535,0.3090169943749475) {};\node (2) at (-0.5877852522924732,-0.8090169943749473) {};\node (3) at (0.5877852522924729,-0.8090169943749476) {};\node (4) at (0.9510565162951536,0.3090169943749472) {};
         
    };
  \end{tikzpicture}\\
24 & \begin{tikzpicture}[>=latex',line join=bevel,scale=.3]
    {
         \tikzstyle{every node}=[draw, circle, inner sep=1pt, minimum size=.1cm]
         \node (0) at (6.123233995736766e-17,1.0) {};\node (1) at (-0.9510565162951535,0.3090169943749475) {};\node (2) at (-0.5877852522924732,-0.8090169943749473) {};\node (3) at (0.5877852522924729,-0.8090169943749476) {};\node (4) at (0.9510565162951536,0.3090169943749472) {};
         \draw (0) -- (2); \draw (0) -- (3); \draw (0) -- (4); \draw (2) -- (3); \draw (2) -- (4); \draw (3) -- (4);
    };
  \end{tikzpicture}\quad\begin{tikzpicture}[>=latex',line join=bevel,scale=.3]
    {
         \tikzstyle{every node}=[draw, circle, inner sep=1pt, minimum size=.1cm]
         \node (0) at (6.123233995736766e-17,1.0) {};\node (1) at (-0.9510565162951535,0.3090169943749475) {};\node (2) at (-0.5877852522924732,-0.8090169943749473) {};\node (3) at (0.5877852522924729,-0.8090169943749476) {};\node (4) at (0.9510565162951536,0.3090169943749472) {};
         \draw (0) -- (4); \draw (1) -- (4); \draw (2) -- (4); \draw (3) -- (4);
    };
  \end{tikzpicture}\\
12 & \begin{tikzpicture}[>=latex',line join=bevel,scale=.3]
    {
         \tikzstyle{every node}=[draw, circle, inner sep=1pt, minimum size=.1cm]
         \node (0) at (6.123233995736766e-17,1.0) {};\node (1) at (-0.9510565162951535,0.3090169943749475) {};\node (2) at (-0.5877852522924732,-0.8090169943749473) {};\node (3) at (0.5877852522924729,-0.8090169943749476) {};\node (4) at (0.9510565162951536,0.3090169943749472) {};
         \draw (0) -- (3); \draw (0) -- (4); \draw (3) -- (4);
    };
  \end{tikzpicture}\quad\begin{tikzpicture}[>=latex',line join=bevel,scale=.3]
    {
         \tikzstyle{every node}=[draw, circle, inner sep=1pt, minimum size=.1cm]
         \node (0) at (6.123233995736766e-17,1.0) {};\node (1) at (-0.9510565162951535,0.3090169943749475) {};\node (2) at (-0.5877852522924732,-0.8090169943749473) {};\node (3) at (0.5877852522924729,-0.8090169943749476) {};\node (4) at (0.9510565162951536,0.3090169943749472) {};
         \draw (0) -- (2); \draw (0) -- (3); \draw (0) -- (4); \draw (1) -- (2); \draw (1) -- (3); \draw (1) -- (4); \draw (2) -- (3); \draw (2) -- (4); \draw (3) -- (4);
    };
  \end{tikzpicture}\quad\begin{tikzpicture}[>=latex',line join=bevel,scale=.3]
    {
         \tikzstyle{every node}=[draw, circle, inner sep=1pt, minimum size=.1cm]
         \node (0) at (6.123233995736766e-17,1.0) {};\node (1) at (-0.9510565162951535,0.3090169943749475) {};\node (2) at (-0.5877852522924732,-0.8090169943749473) {};\node (3) at (0.5877852522924729,-0.8090169943749476) {};\node (4) at (0.9510565162951536,0.3090169943749472) {};
         \draw (0) -- (3); \draw (0) -- (4); \draw (1) -- (3); \draw (1) -- (4); \draw (2) -- (3); \draw (2) -- (4); \draw (3) -- (4);
    };
  \end{tikzpicture}\quad\begin{tikzpicture}[>=latex',line join=bevel,scale=.3]
    {
         \tikzstyle{every node}=[draw, circle, inner sep=1pt, minimum size=.1cm]
         \node (0) at (6.123233995736766e-17,1.0) {};\node (1) at (-0.9510565162951535,0.3090169943749475) {};\node (2) at (-0.5877852522924732,-0.8090169943749473) {};\node (3) at (0.5877852522924729,-0.8090169943749476) {};\node (4) at (0.9510565162951536,0.3090169943749472) {};
         \draw (0) -- (3); \draw (0) -- (4); \draw (1) -- (3); \draw (1) -- (4); \draw (2) -- (3); \draw (2) -- (4);
    };
  \end{tikzpicture}\quad\begin{tikzpicture}[>=latex',line join=bevel,scale=.3]
    {
         \tikzstyle{every node}=[draw, circle, inner sep=1pt, minimum size=.1cm]
         \node (0) at (6.123233995736766e-17,1.0) {};\node (1) at (-0.9510565162951535,0.3090169943749475) {};\node (2) at (-0.5877852522924732,-0.8090169943749473) {};\node (3) at (0.5877852522924729,-0.8090169943749476) {};\node (4) at (0.9510565162951536,0.3090169943749472) {};
         \draw (0) -- (2); \draw (0) -- (4); \draw (1) -- (3); \draw (2) -- (4);
    };
  \end{tikzpicture}\quad\begin{tikzpicture}[>=latex',line join=bevel,scale=.3]
    {
         \tikzstyle{every node}=[draw, circle, inner sep=1pt, minimum size=.1cm]
         \node (0) at (6.123233995736766e-17,1.0) {};\node (1) at (-0.9510565162951535,0.3090169943749475) {};\node (2) at (-0.5877852522924732,-0.8090169943749473) {};\node (3) at (0.5877852522924729,-0.8090169943749476) {};\node (4) at (0.9510565162951536,0.3090169943749472) {};
         \draw (0) -- (4);
    };
  \end{tikzpicture}\\
10 & \begin{tikzpicture}[>=latex',line join=bevel,scale=.3]
    {
         \tikzstyle{every node}=[draw, circle, inner sep=1pt, minimum size=.1cm]
         \node (0) at (6.123233995736766e-17,1.0) {};\node (1) at (-0.9510565162951535,0.3090169943749475) {};\node (2) at (-0.5877852522924732,-0.8090169943749473) {};\node (3) at (0.5877852522924729,-0.8090169943749476) {};\node (4) at (0.9510565162951536,0.3090169943749472) {};
         \draw (0) -- (2); \draw (0) -- (3); \draw (1) -- (3); \draw (1) -- (4); \draw (2) -- (4);
    };
  \end{tikzpicture}\\
8 & \begin{tikzpicture}[>=latex',line join=bevel,scale=.3]
    {
         \tikzstyle{every node}=[draw, circle, inner sep=1pt, minimum size=.1cm]
         \node (0) at (6.123233995736766e-17,1.0) {};\node (1) at (-0.9510565162951535,0.3090169943749475) {};\node (2) at (-0.5877852522924732,-0.8090169943749473) {};\node (3) at (0.5877852522924729,-0.8090169943749476) {};\node (4) at (0.9510565162951536,0.3090169943749472) {};
         \draw (0) -- (2); \draw (0) -- (3); \draw (0) -- (4); \draw (1) -- (2); \draw (1) -- (3); \draw (1) -- (4); \draw (2) -- (4); \draw (3) -- (4);
    };
  \end{tikzpicture}\quad\begin{tikzpicture}[>=latex',line join=bevel,scale=.3]
    {
         \tikzstyle{every node}=[draw, circle, inner sep=1pt, minimum size=.1cm]
         \node (0) at (6.123233995736766e-17,1.0) {};\node (1) at (-0.9510565162951535,0.3090169943749475) {};\node (2) at (-0.5877852522924732,-0.8090169943749473) {};\node (3) at (0.5877852522924729,-0.8090169943749476) {};\node (4) at (0.9510565162951536,0.3090169943749472) {};
         \draw (0) -- (3); \draw (0) -- (4); \draw (1) -- (3); \draw (1) -- (4);
    };
  \end{tikzpicture}\quad\begin{tikzpicture}[>=latex',line join=bevel,scale=.3]
    {
         \tikzstyle{every node}=[draw, circle, inner sep=1pt, minimum size=.1cm]
         \node (0) at (6.123233995736766e-17,1.0) {};\node (1) at (-0.9510565162951535,0.3090169943749475) {};\node (2) at (-0.5877852522924732,-0.8090169943749473) {};\node (3) at (0.5877852522924729,-0.8090169943749476) {};\node (4) at (0.9510565162951536,0.3090169943749472) {};
         \draw (0) -- (2); \draw (0) -- (4); \draw (1) -- (3); \draw (1) -- (4); \draw (2) -- (4); \draw (3) -- (4);
    };
  \end{tikzpicture}\quad\begin{tikzpicture}[>=latex',line join=bevel,scale=.3]
    {
         \tikzstyle{every node}=[draw, circle, inner sep=1pt, minimum size=.1cm]
         \node (0) at (6.123233995736766e-17,1.0) {};\node (1) at (-0.9510565162951535,0.3090169943749475) {};\node (2) at (-0.5877852522924732,-0.8090169943749473) {};\node (3) at (0.5877852522924729,-0.8090169943749476) {};\node (4) at (0.9510565162951536,0.3090169943749472) {};
         \draw (0) -- (3); \draw (1) -- (4);
    };
  \end{tikzpicture}\\
6 & \begin{tikzpicture}[>=latex',line join=bevel,scale=.3]
    {
         \tikzstyle{every node}=[draw, circle, inner sep=1pt, minimum size=.1cm]
         \node (0) at (6.123233995736766e-17,1.0) {};\node (1) at (-0.9510565162951535,0.3090169943749475) {};\node (2) at (-0.5877852522924732,-0.8090169943749473) {};\node (3) at (0.5877852522924729,-0.8090169943749476) {};\node (4) at (0.9510565162951536,0.3090169943749472) {};
         \draw (0) -- (2); \draw (0) -- (3); \draw (0) -- (4); \draw (1) -- (4); \draw (2) -- (3); \draw (2) -- (4); \draw (3) -- (4);
    };
  \end{tikzpicture}\quad\begin{tikzpicture}[>=latex',line join=bevel,scale=.3]
    {
         \tikzstyle{every node}=[draw, circle, inner sep=1pt, minimum size=.1cm]
         \node (0) at (6.123233995736766e-17,1.0) {};\node (1) at (-0.9510565162951535,0.3090169943749475) {};\node (2) at (-0.5877852522924732,-0.8090169943749473) {};\node (3) at (0.5877852522924729,-0.8090169943749476) {};\node (4) at (0.9510565162951536,0.3090169943749472) {};
         \draw (0) -- (4); \draw (1) -- (4); \draw (2) -- (4);
    };
  \end{tikzpicture}\\
4 & \begin{tikzpicture}[>=latex',line join=bevel,scale=.3]
    {
         \tikzstyle{every node}=[draw, circle, inner sep=1pt, minimum size=.1cm]
         \node (0) at (6.123233995736766e-17,1.0) {};\node (1) at (-0.9510565162951535,0.3090169943749475) {};\node (2) at (-0.5877852522924732,-0.8090169943749473) {};\node (3) at (0.5877852522924729,-0.8090169943749476) {};\node (4) at (0.9510565162951536,0.3090169943749472) {};
         \draw (0) -- (2); \draw (0) -- (3); \draw (0) -- (4); \draw (1) -- (2); \draw (1) -- (3); \draw (1) -- (4); \draw (2) -- (4);
    };
  \end{tikzpicture}\quad\begin{tikzpicture}[>=latex',line join=bevel,scale=.3]
    {
         \tikzstyle{every node}=[draw, circle, inner sep=1pt, minimum size=.1cm]
         \node (0) at (6.123233995736766e-17,1.0) {};\node (1) at (-0.9510565162951535,0.3090169943749475) {};\node (2) at (-0.5877852522924732,-0.8090169943749473) {};\node (3) at (0.5877852522924729,-0.8090169943749476) {};\node (4) at (0.9510565162951536,0.3090169943749472) {};
         \draw (0) -- (4); \draw (1) -- (4);
    };
  \end{tikzpicture}\quad\begin{tikzpicture}[>=latex',line join=bevel,scale=.3]
    {
         \tikzstyle{every node}=[draw, circle, inner sep=1pt, minimum size=.1cm]
         \node (0) at (6.123233995736766e-17,1.0) {};\node (1) at (-0.9510565162951535,0.3090169943749475) {};\node (2) at (-0.5877852522924732,-0.8090169943749473) {};\node (3) at (0.5877852522924729,-0.8090169943749476) {};\node (4) at (0.9510565162951536,0.3090169943749472) {};
         \draw (0) -- (2); \draw (0) -- (3); \draw (0) -- (4); \draw (1) -- (3); \draw (1) -- (4); \draw (2) -- (3); \draw (2) -- (4); \draw (3) -- (4);
    };
  \end{tikzpicture}\quad\begin{tikzpicture}[>=latex',line join=bevel,scale=.3]
    {
         \tikzstyle{every node}=[draw, circle, inner sep=1pt, minimum size=.1cm]
         \node (0) at (6.123233995736766e-17,1.0) {};\node (1) at (-0.9510565162951535,0.3090169943749475) {};\node (2) at (-0.5877852522924732,-0.8090169943749473) {};\node (3) at (0.5877852522924729,-0.8090169943749476) {};\node (4) at (0.9510565162951536,0.3090169943749472) {};
         \draw (0) -- (3); \draw (0) -- (4); \draw (1) -- (3); \draw (1) -- (4); \draw (3) -- (4);
    };
  \end{tikzpicture}\quad\begin{tikzpicture}[>=latex',line join=bevel,scale=.3]
    {
         \tikzstyle{every node}=[draw, circle, inner sep=1pt, minimum size=.1cm]
         \node (0) at (6.123233995736766e-17,1.0) {};\node (1) at (-0.9510565162951535,0.3090169943749475) {};\node (2) at (-0.5877852522924732,-0.8090169943749473) {};\node (3) at (0.5877852522924729,-0.8090169943749476) {};\node (4) at (0.9510565162951536,0.3090169943749472) {};
         \draw (0) -- (3); \draw (0) -- (4); \draw (1) -- (4); \draw (2) -- (4); \draw (3) -- (4);
    };
  \end{tikzpicture}\quad\begin{tikzpicture}[>=latex',line join=bevel,scale=.3]
    {
         \tikzstyle{every node}=[draw, circle, inner sep=1pt, minimum size=.1cm]
         \node (0) at (6.123233995736766e-17,1.0) {};\node (1) at (-0.9510565162951535,0.3090169943749475) {};\node (2) at (-0.5877852522924732,-0.8090169943749473) {};\node (3) at (0.5877852522924729,-0.8090169943749476) {};\node (4) at (0.9510565162951536,0.3090169943749472) {};
         \draw (0) -- (3); \draw (1) -- (4); \draw (2) -- (4);
    };
  \end{tikzpicture}\\
2 & \begin{tikzpicture}[>=latex',line join=bevel,scale=.3]
    {
         \tikzstyle{every node}=[draw, circle, inner sep=1pt, minimum size=.1cm]
         \node (0) at (6.123233995736766e-17,1.0) {};\node (1) at (-0.9510565162951535,0.3090169943749475) {};\node (2) at (-0.5877852522924732,-0.8090169943749473) {};\node (3) at (0.5877852522924729,-0.8090169943749476) {};\node (4) at (0.9510565162951536,0.3090169943749472) {};
         \draw (0) -- (2); \draw (0) -- (3); \draw (0) -- (4); \draw (1) -- (3); \draw (1) -- (4); \draw (2) -- (4); \draw (3) -- (4);
    };
  \end{tikzpicture}\quad\begin{tikzpicture}[>=latex',line join=bevel,scale=.3]
    {
         \tikzstyle{every node}=[draw, circle, inner sep=1pt, minimum size=.1cm]
         \node (0) at (6.123233995736766e-17,1.0) {};\node (1) at (-0.9510565162951535,0.3090169943749475) {};\node (2) at (-0.5877852522924732,-0.8090169943749473) {};\node (3) at (0.5877852522924729,-0.8090169943749476) {};\node (4) at (0.9510565162951536,0.3090169943749472) {};
         \draw (0) -- (2); \draw (0) -- (3); \draw (0) -- (4); \draw (1) -- (3); \draw (1) -- (4); \draw (2) -- (4);
    };
  \end{tikzpicture}\quad\begin{tikzpicture}[>=latex',line join=bevel,scale=.3]
    {
         \tikzstyle{every node}=[draw, circle, inner sep=1pt, minimum size=.1cm]
         \node (0) at (6.123233995736766e-17,1.0) {};\node (1) at (-0.9510565162951535,0.3090169943749475) {};\node (2) at (-0.5877852522924732,-0.8090169943749473) {};\node (3) at (0.5877852522924729,-0.8090169943749476) {};\node (4) at (0.9510565162951536,0.3090169943749472) {};
         \draw (0) -- (2); \draw (0) -- (3); \draw (0) -- (4); \draw (1) -- (4); \draw (2) -- (3); \draw (2) -- (4);
    };
  \end{tikzpicture}\quad\begin{tikzpicture}[>=latex',line join=bevel,scale=.3]
    {
         \tikzstyle{every node}=[draw, circle, inner sep=1pt, minimum size=.1cm]
         \node (0) at (6.123233995736766e-17,1.0) {};\node (1) at (-0.9510565162951535,0.3090169943749475) {};\node (2) at (-0.5877852522924732,-0.8090169943749473) {};\node (3) at (0.5877852522924729,-0.8090169943749476) {};\node (4) at (0.9510565162951536,0.3090169943749472) {};
         \draw (0) -- (2); \draw (0) -- (4); \draw (1) -- (3); \draw (1) -- (4); \draw (2) -- (4);
    };
  \end{tikzpicture}\quad\begin{tikzpicture}[>=latex',line join=bevel,scale=.3]
    {
         \tikzstyle{every node}=[draw, circle, inner sep=1pt, minimum size=.1cm]
         \node (0) at (6.123233995736766e-17,1.0) {};\node (1) at (-0.9510565162951535,0.3090169943749475) {};\node (2) at (-0.5877852522924732,-0.8090169943749473) {};\node (3) at (0.5877852522924729,-0.8090169943749476) {};\node (4) at (0.9510565162951536,0.3090169943749472) {};
         \draw (0) -- (3); \draw (0) -- (4); \draw (1) -- (4); \draw (2) -- (4);
    };
  \end{tikzpicture}\quad\begin{tikzpicture}[>=latex',line join=bevel,scale=.3]
    {
         \tikzstyle{every node}=[draw, circle, inner sep=1pt, minimum size=.1cm]
         \node (0) at (6.123233995736766e-17,1.0) {};\node (1) at (-0.9510565162951535,0.3090169943749475) {};\node (2) at (-0.5877852522924732,-0.8090169943749473) {};\node (3) at (0.5877852522924729,-0.8090169943749476) {};\node (4) at (0.9510565162951536,0.3090169943749472) {};
         \draw (0) -- (2); \draw (0) -- (4); \draw (1) -- (3); \draw (1) -- (4);
    };
  \end{tikzpicture}\quad\begin{tikzpicture}[>=latex',line join=bevel,scale=.3]
    {
         \tikzstyle{every node}=[draw, circle, inner sep=1pt, minimum size=.1cm]
         \node (0) at (6.123233995736766e-17,1.0) {};\node (1) at (-0.9510565162951535,0.3090169943749475) {};\node (2) at (-0.5877852522924732,-0.8090169943749473) {};\node (3) at (0.5877852522924729,-0.8090169943749476) {};\node (4) at (0.9510565162951536,0.3090169943749472) {};
         \draw (0) -- (3); \draw (0) -- (4); \draw (1) -- (4);
    };
  \end{tikzpicture}\quad\begin{tikzpicture}[>=latex',line join=bevel,scale=.3]
    {
         \tikzstyle{every node}=[draw, circle, inner sep=1pt, minimum size=.1cm]
         \node (0) at (6.123233995736766e-17,1.0) {};\node (1) at (-0.9510565162951535,0.3090169943749475) {};\node (2) at (-0.5877852522924732,-0.8090169943749473) {};\node (3) at (0.5877852522924729,-0.8090169943749476) {};\node (4) at (0.9510565162951536,0.3090169943749472) {};
         \draw (0) -- (3); \draw (0) -- (4); \draw (1) -- (4); \draw (3) -- (4);
    };
  \end{tikzpicture}\quad\begin{tikzpicture}[>=latex',line join=bevel,scale=.3]
    {
         \tikzstyle{every node}=[draw, circle, inner sep=1pt, minimum size=.1cm]
         \node (0) at (6.123233995736766e-17,1.0) {};\node (1) at (-0.9510565162951535,0.3090169943749475) {};\node (2) at (-0.5877852522924732,-0.8090169943749473) {};\node (3) at (0.5877852522924729,-0.8090169943749476) {};\node (4) at (0.9510565162951536,0.3090169943749472) {};
         \draw (0) -- (3); \draw (0) -- (4); \draw (1) -- (3); \draw (1) -- (4); \draw (2) -- (4);
    };
  \end{tikzpicture}\quad\begin{tikzpicture}[>=latex',line join=bevel,scale=.3]
    {
         \tikzstyle{every node}=[draw, circle, inner sep=1pt, minimum size=.1cm]
         \node (0) at (6.123233995736766e-17,1.0) {};\node (1) at (-0.9510565162951535,0.3090169943749475) {};\node (2) at (-0.5877852522924732,-0.8090169943749473) {};\node (3) at (0.5877852522924729,-0.8090169943749476) {};\node (4) at (0.9510565162951536,0.3090169943749472) {};
         \draw (0) -- (3); \draw (0) -- (4); \draw (1) -- (3); \draw (1) -- (4); \draw (2) -- (4); \draw (3) -- (4);
    };
  \end{tikzpicture}\quad\begin{tikzpicture}[>=latex',line join=bevel,scale=.3]
    {
         \tikzstyle{every node}=[draw, circle, inner sep=1pt, minimum size=.1cm]
         \node (0) at (6.123233995736766e-17,1.0) {};\node (1) at (-0.9510565162951535,0.3090169943749475) {};\node (2) at (-0.5877852522924732,-0.8090169943749473) {};\node (3) at (0.5877852522924729,-0.8090169943749476) {};\node (4) at (0.9510565162951536,0.3090169943749472) {};
         \draw (0) -- (3); \draw (0) -- (4); \draw (1) -- (3); \draw (2) -- (4); \draw (3) -- (4);
    };
  \end{tikzpicture}\\

 \bottomrule
\end{tabular}
\caption{The 34 graphs of order 5, classified by automorphism group size}
\label{tab:graphs-by-autom-count}
\end{table}

\subsection{Testing the strategies: $\calF \subset \calF(5)$}

While the almost-random strategy was very effective when searching
for induced universal graphs of \emph{all} graphs of a given order,
it is plausible that it will be less effective if $\calF$ is a strict
subset of $\calF(k)$, and in particular if it includes
neither $I_k$ nor $K_k$.  \Cref{fig:second-experiment-using-sample}
examines this claim.  We generated 50 subsets $\calF$ of $\calF(5)$ by random selection,
with each subset containing exactly half of the elements of $\calF(5)$.
For each of these, we ran $\FuncSty{AllInducedUniversalGraphs}(\calF,9)$
using each of the four strategies.  The `automorphisms' strategy
was the best overall in this case, requiring on average $8\%$, $27\%$, and $52\%$
fewer calls to the subgraph isomorphism function than the edges, almost-random,
and random strategies respectively.

To help us understand why the almost-random
strategy performs poorly for these instances, we connect the four solver runs on
each instance together on the plot with blue lines.  The almost-random
strategy, on the instances where it performs poorly, requires
exactly the same number of subgraph isomorphism calls as the random strategy.
These are the instances in which $\calF$ contains neither $K_5$ nor $I_5$;
thus, the almost-random and random strategies are equivalent.  Notably,
for these instances it is also the case that the automorphisms strategy
tends to outperform the edges strategy.

\begin{figure}[htb]
    \centering

    \includegraphics*[width=0.7\textwidth]{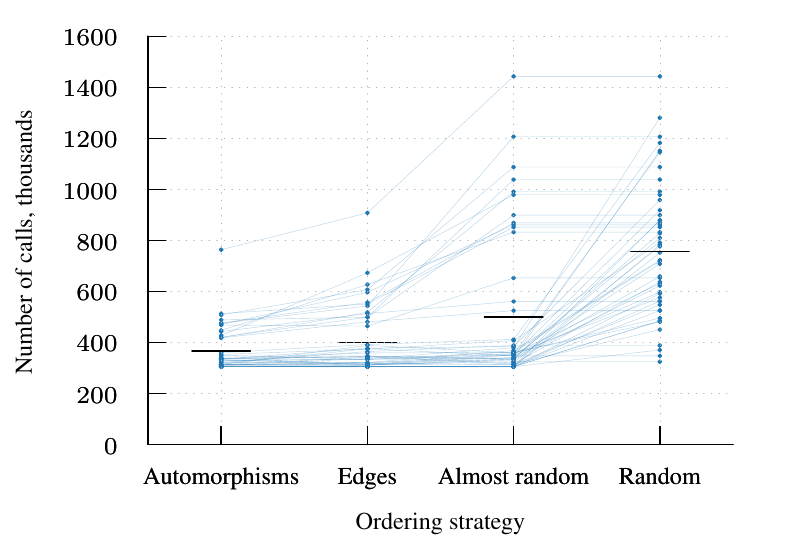}

    \caption{Using each of our four order strategies, we ran
        $\FuncSty{AllInducedUniversalGraphs}(\calF,9)$ 50 times, where
        $\calF$ is a random subset of $\calF(5)$ containing 17 graphs. This
        parallel coordinates plot
        shows the number calls to the subgraph isomorphism algorithm, in thousands.
        Horizontal black lines show average (mean) number of calls.  The four runs for
        each instance are connected by blue lines.}
\label{fig:second-experiment-using-sample}
\end{figure}

\subsection{Testing the strategies: a family of trees}

As a final test case, we let $\calF$ be the family of all trees on 6 vertices,
and searched for all graphs of order 8 that are induced universal for $\calF$.
\Cref{fig:second-experiment-using-trees} shows the results. The edges, almost-random,
and random strategies had identical results, since all members of $\calF$ have five
edges and the family contains neither $K_6$ nor $I_6$.  The automorphisms strategy
requires fewer subgraph isomorphism calls on average than the other three strategies,
although on seven of the 50 instances it does require marginally more calls than
the other strategies.

\begin{figure}[htb]
    \centering

    \includegraphics*[width=0.7\textwidth]{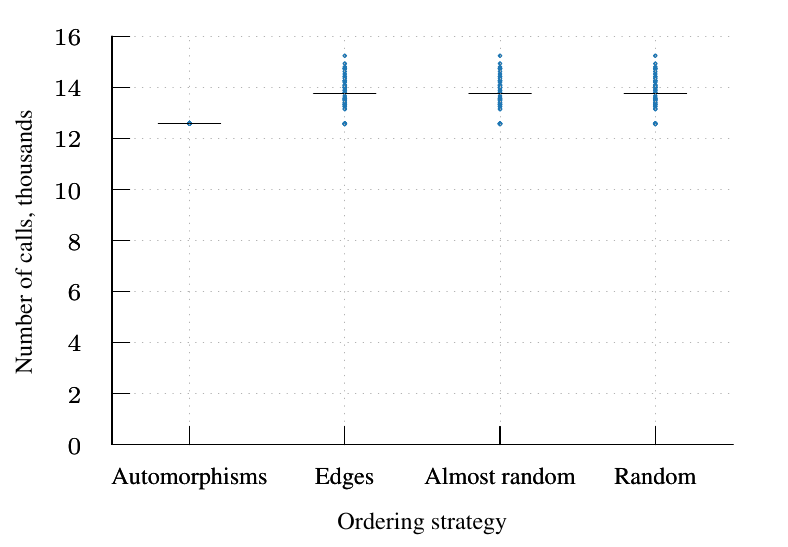}

    \caption{Using each of our four order strategies, we ran
        $\FuncSty{AllInducedUniversalGraphs}(\calF,8)$ 50 times, where
        $\calF$ is the set of all trees of order 6. The plot
        shows the number calls to the subgraph isomorphism algorithm, in thousands.
        Horizontal black lines show average (mean) number of calls.}
\label{fig:second-experiment-using-trees}
\end{figure}

The `automorphisms' strategy performs consistently well across our three test
cases; therefore, we use this strategy for the remainder of the paper.

\section{Results for \texorpdfstring{$k \leq 5$}{k<=5}}\label{sec:results5}

Recall that $\calF(k)$ is the family of all graphs of order $k$, $f(k)$ is
the order of a minimal induced universal graph for
$\calF(k)$, and $F(k)$ is the number of non-isomorphic graphs of that order
that are induced universal for $\calF(k)$.  This section
presents values of $f(k)$ and $F(k)$ for $0 \leq k \leq 5$, computed
using \Cref{alg:brute-force}. The results are shown in 
Table~\ref{tab:graphresults}.  To my knowledge, the value of
$f(5)$ and the values of $F(k)$ have not been published previously.

The values $f(1)$, $f(2)$ and $f(3)$ are equal to to the simple lower bound $2k
- 1$ given in a question by ``Chain Markov'' on Mathematics Stack Exchange
  \cite{math_se_question}.  The values $f(4)$ and $f(5)$ are equal to a lower
  bound given in a comment by ``bof'' on the same question: $f(k) \geq 2k$ if $k
  \geq 4$.  (For $k < 10$, this bound improves upon Moon's lower bound $f(k)
  \leq 2^{(k-1)/2}$.) To briefly summarise the proof, if $f(k) \leq 2k$ then $G$
  must be a split graph (that is, a graph whose vertices can be partitioned
  into a clique and an independent set); therefore, $G$ cannot contain the
  cycle $C_4$ as an induced subgraph.  An example of an 8-vertex induced
  universal graph for the family of 4-vertex graphs was given by ``Chain
  Markov'' as a comment on the same question.

\begin{table}[h!]
\centering
\begin{tabular}{r r r}
 \toprule
 $k$ & $f(k)$ & $F(k)$ \\ [0.5ex]
 \midrule
 0 & 0 & 1 \\
 1 & 1 & 1 \\
 2 & 3 & 2 \\
 3 & 5 & 5 \\
 4 & 8 & 438 \\
 5 & 10 & 22 \\
 \bottomrule
\end{tabular}
\caption{For each $k$, $f(k)$ is the minimum order of a graph containing all $k$-vertex graphs as
induced subgraphs, and $F(k)$ is the number of distinct $f(k)$-vertex graphs that contain
all $k$-vertex graphs as induced subgraphs.}
\label{tab:graphresults}
\end{table}

Figure~\ref{fig:graphs} shows examples of minimal induced universal graphs
for the families of all graphs with four and five vertices
respectively.

\begin{figure}[htb]
    \centering

\begin{tikzpicture}[>=latex',line join=bevel,scale=.4]
  \pgfsetlinewidth{.5bp}
\pgfsetcolor{black}
  \draw [] (44.956bp,147.01bp) .. controls (40.082bp,141.1bp) and (34.42bp,134.24bp)  .. (29.546bp,128.33bp);
  \draw [] (166.49bp,80.229bp) .. controls (168.44bp,90.759bp) and (170.89bp,103.99bp)  .. (172.84bp,114.47bp);
  \draw [] (81.617bp,113.74bp) .. controls (68.067bp,113.87bp) and (49.632bp,114.04bp)  .. (36.095bp,114.17bp);
  \draw [] (117.47bp,117.89bp) .. controls (129.81bp,120.9bp) and (146.19bp,124.9bp)  .. (158.51bp,127.9bp);
  \draw [] (96.884bp,95.599bp) .. controls (94.148bp,78.511bp) and (90.056bp,52.957bp)  .. (87.325bp,35.896bp);
  \draw [] (87.484bp,127.04bp) .. controls (81.59bp,133.51bp) and (74.549bp,141.23bp)  .. (68.67bp,147.68bp);
  \draw [] (56.422bp,85.611bp) .. controls (48.962bp,91.188bp) and (39.886bp,97.973bp)  .. (32.429bp,103.55bp);
  \draw [] (75.094bp,57.147bp) .. controls (76.715bp,50.372bp) and (78.564bp,42.641bp)  .. (80.189bp,35.849bp);
  \draw [] (67.894bp,92.673bp) .. controls (65.43bp,107.47bp) and (61.945bp,128.39bp)  .. (59.482bp,143.18bp);
  \draw [] (89.1bp,72.32bp) .. controls (105.38bp,70.096bp) and (129.09bp,66.857bp)  .. (145.27bp,64.648bp);
  \draw [] (81.628bp,89.246bp) .. controls (84.002bp,92.431bp) and (86.52bp,95.808bp)  .. (88.898bp,99.0bp);
\begin{scope}
  \definecolor{strokecol}{rgb}{0.0,0.0,0.0};
  \pgfsetstrokecolor{strokecol}
  \draw (18.0bp,114.33bp) ellipse (18.0bp and 18.0bp);
\end{scope}
\begin{scope}
  \definecolor{strokecol}{rgb}{0.0,0.0,0.0};
  \pgfsetstrokecolor{strokecol}
  \draw (176.12bp,132.19bp) ellipse (18.0bp and 18.0bp);
\end{scope}
\begin{scope}
  \definecolor{strokecol}{rgb}{0.0,0.0,0.0};
  \pgfsetstrokecolor{strokecol}
  \draw (84.46bp,18.0bp) ellipse (18.0bp and 18.0bp);
\end{scope}
\begin{scope}
  \definecolor{strokecol}{rgb}{0.0,0.0,0.0};
  \pgfsetstrokecolor{strokecol}
  \draw (210.0bp,194.0bp) ellipse (18.0bp and 18.0bp);
\end{scope}
\begin{scope}
  \definecolor{strokecol}{rgb}{0.0,0.0,0.0};
  \pgfsetstrokecolor{strokecol}
  \draw (56.51bp,161.02bp) ellipse (18.0bp and 18.0bp);
\end{scope}
\begin{scope}
  \definecolor{strokecol}{rgb}{0.0,0.0,0.0};
  \pgfsetstrokecolor{strokecol}
  \draw (163.15bp,62.21bp) ellipse (18.0bp and 18.0bp);
\end{scope}
\begin{scope}
  \definecolor{strokecol}{rgb}{0.0,0.0,0.0};
  \pgfsetstrokecolor{strokecol}
  \draw (99.76bp,113.58bp) ellipse (18.0bp and 18.0bp);
\end{scope}
\begin{scope}
  \definecolor{strokecol}{rgb}{0.0,0.0,0.0};
  \pgfsetstrokecolor{strokecol}
  \draw (70.87bp,74.81bp) ellipse (18.0bp and 18.0bp);
\end{scope}
\end{tikzpicture}
\qquad
\qquad
\begin{tikzpicture}[>=latex',line join=bevel,scale=.4]
  \pgfsetlinewidth{.5bp}
\pgfsetcolor{black}
  \draw [] (50.975bp,158.84bp) .. controls (45.582bp,156.23bp) and (39.597bp,153.34bp)  .. (34.222bp,150.74bp);
  \draw [] (84.643bp,161.37bp) .. controls (96.964bp,157.48bp) and (113.5bp,152.27bp)  .. (125.86bp,148.36bp);
  \draw [] (21.443bp,105.03bp) .. controls (20.863bp,111.4bp) and (20.211bp,118.58bp)  .. (19.631bp,124.95bp);
  \draw [] (31.856bp,71.036bp) .. controls (37.972bp,59.88bp) and (46.093bp,45.065bp)  .. (52.198bp,33.928bp);
  \draw [] (31.841bp,102.81bp) .. controls (39.568bp,116.7bp) and (50.773bp,136.84bp)  .. (58.539bp,150.8bp);
  \draw [] (135.92bp,66.785bp) .. controls (137.53bp,83.477bp) and (139.91bp,108.08bp)  .. (141.51bp,124.75bp);
  \draw [] (117.55bp,41.664bp) .. controls (105.65bp,36.69bp) and (89.681bp,30.016bp)  .. (77.737bp,25.024bp);
  \draw [] (146.34bp,62.563bp) .. controls (149.79bp,66.518bp) and (153.54bp,70.819bp)  .. (156.97bp,74.761bp);
  \draw [] (79.976bp,126.94bp) .. controls (66.776bp,130.34bp) and (48.816bp,134.96bp)  .. (35.628bp,138.36bp);
  \draw [] (114.12bp,129.78bp) .. controls (118.2bp,131.62bp) and (122.58bp,133.58bp)  .. (126.67bp,135.42bp);
  \draw [] (114.19bp,114.53bp) .. controls (125.71bp,109.05bp) and (141.01bp,101.78bp)  .. (152.51bp,96.32bp);
  \draw [] (87.424bp,137.42bp) .. controls (84.273bp,142.05bp) and (80.806bp,147.15bp)  .. (77.657bp,151.78bp);
  \draw [] (81.103bp,114.54bp) .. controls (68.744bp,108.69bp) and (51.929bp,100.72bp)  .. (39.582bp,94.871bp);
  \draw [] (105.76bp,106.01bp) .. controls (111.81bp,93.786bp) and (120.04bp,77.152bp)  .. (126.09bp,64.937bp);
  \draw [] (66.194bp,101.08bp) .. controls (55.929bp,109.98bp) and (41.964bp,122.1bp)  .. (31.708bp,131.0bp);
  \draw [] (93.68bp,100.8bp) .. controls (104.16bp,109.69bp) and (118.53bp,121.89bp)  .. (129.11bp,130.87bp);
  \draw [] (75.241bp,71.56bp) .. controls (72.339bp,60.696bp) and (68.624bp,46.796bp)  .. (65.707bp,35.879bp);
  \draw [] (97.955bp,89.017bp) .. controls (113.3bp,88.906bp) and (135.24bp,88.748bp)  .. (150.67bp,88.636bp);
  \draw [] (77.041bp,107.16bp) .. controls (75.021bp,119.7bp) and (72.338bp,136.36bp)  .. (70.321bp,148.89bp);
  \draw [] (61.949bp,88.484bp) .. controls (55.311bp,88.239bp) and (47.792bp,87.962bp)  .. (41.147bp,87.717bp);
  \draw [] (88.59bp,105.38bp) .. controls (88.73bp,105.64bp) and (88.871bp,105.9bp)  .. (89.011bp,106.17bp);
\begin{scope}
  \definecolor{strokecol}{rgb}{0.0,0.0,0.0};
  \pgfsetstrokecolor{strokecol}
  \draw (18.0bp,142.9bp) ellipse (18.0bp and 18.0bp);
\end{scope}
\begin{scope}
  \definecolor{strokecol}{rgb}{0.0,0.0,0.0};
  \pgfsetstrokecolor{strokecol}
  \draw (143.26bp,142.87bp) ellipse (18.0bp and 18.0bp);
\end{scope}
\begin{scope}
  \definecolor{strokecol}{rgb}{0.0,0.0,0.0};
  \pgfsetstrokecolor{strokecol}
  \draw (60.93bp,18.0bp) ellipse (18.0bp and 18.0bp);
\end{scope}
\begin{scope}
  \definecolor{strokecol}{rgb}{0.0,0.0,0.0};
  \pgfsetstrokecolor{strokecol}
  \draw (168.96bp,88.5bp) ellipse (18.0bp and 18.0bp);
\end{scope}
\begin{scope}
  \definecolor{strokecol}{rgb}{0.0,0.0,0.0};
  \pgfsetstrokecolor{strokecol}
  \draw (210.0bp,18.0bp) ellipse (18.0bp and 18.0bp);
\end{scope}
\begin{scope}
  \definecolor{strokecol}{rgb}{0.0,0.0,0.0};
  \pgfsetstrokecolor{strokecol}
  \draw (67.44bp,166.8bp) ellipse (18.0bp and 18.0bp);
\end{scope}
\begin{scope}
  \definecolor{strokecol}{rgb}{0.0,0.0,0.0};
  \pgfsetstrokecolor{strokecol}
  \draw (23.08bp,87.05bp) ellipse (18.0bp and 18.0bp);
\end{scope}
\begin{scope}
  \definecolor{strokecol}{rgb}{0.0,0.0,0.0};
  \pgfsetstrokecolor{strokecol}
  \draw (134.17bp,48.61bp) ellipse (18.0bp and 18.0bp);
\end{scope}
\begin{scope}
  \definecolor{strokecol}{rgb}{0.0,0.0,0.0};
  \pgfsetstrokecolor{strokecol}
  \draw (97.65bp,122.39bp) ellipse (18.0bp and 18.0bp);
\end{scope}
\begin{scope}
  \definecolor{strokecol}{rgb}{0.0,0.0,0.0};
  \pgfsetstrokecolor{strokecol}
  \draw (79.94bp,89.15bp) ellipse (18.0bp and 18.0bp);
\end{scope}
\end{tikzpicture}
\caption{Examples of minimal induced universal graphs for $\calF(4)$ and $\calF(5)$}
\label{fig:graphs}
\end{figure}
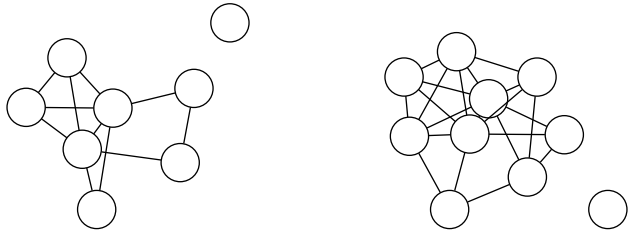

\section{\texorpdfstring{$f(6) = 14$}{f(6)=14}}\label{sec:f6}

This section shows that $f(6) = 14$.  We begin with the lower bound.
For $k \geq 6$, we can increase by 2 the lower bound --- $f(k) \geq 2k$
for $k \geq 4$ --- cited in the previous section.

\begin{proposition}\label{f6proposition}
    $f(k) \geq 2k + 2$ for all $k \geq 6$.
\end{proposition}
\begin{proof}

    Suppose that $G$ is an induced universal graph for the family of all graphs
    on $k$ vertices, and that $G$ has no more than $2k + 1$ vertices.  Graphs
    $K_k$ and $I_k$ are both members of $\calF(k)$, and therefore must
    be induced sugraphs of $G$.  Clearly, this clique and
    independent set may overlap by no more than one vertex, so their union must
    contain at least $2k - 1$ vertices.  Therefore it is possible to partition
    the vertices of $G$ into three sets: a clique $S_1$, an independent set
    $S_2$, and a third set $S_3$ containing at most 2 vertices.

    We will show that $G$ cannot contain as induced subgraphs both
    of the graphs in Figure~\ref{fig:boundproof}.  These graphs are complements
    of each other, and we refer to them as $H$ and $H'$ respectively.

\begin{figure}[htb]
    \centering
\begin{tikzpicture}[>=latex',line join=bevel]
  \tikzstyle{every node}=[draw, circle, inner sep=1pt, minimum size=.5cm]
  \pgfsetlinewidth{.5bp}
  \node at (0,2.9) (1) {};
  \node at (0,2) (2) {};
  \node at (0,1.1) (3) {};
  \node at (1.1,2.9) (4) {};
  \node at (1.1,2) (5) {};
  \node at (1.1,1.1) (6) {};
  \draw (1) -- (2) -- (3);
  \draw (3) to [out=140,in=220] (1);
  \draw (4) -- (5) -- (6);
  \draw (6) to [out=140,in=220] (4);
\end{tikzpicture}
\qquad \qquad
\begin{tikzpicture}[>=latex',line join=bevel]
  \tikzstyle{every node}=[draw, circle, inner sep=1pt, minimum size=.5cm]
  \pgfsetlinewidth{.5bp}
  \node at (0,2.9) (1) {};
  \node at (0,2) (2) {};
  \node at (0,1.1) (3) {};
  \node at (1.1,2.9) (4) {};
  \node at (1.1,2) (5) {};
  \node at (1.1,1.1) (6) {};
  \draw (1) -- (4);
  \draw (1) -- (5);
  \draw (1) -- (6);
  \draw (2) -- (4);
  \draw (2) -- (5);
  \draw (2) -- (6);
  \draw (3) -- (4);
  \draw (3) -- (5);
  \draw (3) -- (6);
\end{tikzpicture}
\caption{If $k > 5$ and $G$ is a graph on $2k + 1$ or fewer vertices, then
$G$ cannot have as induced subgraphs all four of the following graphs: the
clique $K_k$, the independent set $I_k$, and
the two graphs shown. See the the proof of Proposition~\ref{f6proposition}.}
\label{fig:boundproof}
\end{figure}
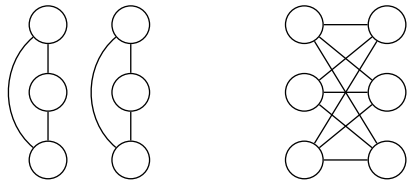

    Let $G_1$ be an induced
    subgraph of $G$ that is isomorphic to $H$.  Since there are
    no edges between the two three-vertex cliques in $G_1$, it must be the case that the
    vertex set of at least one of these cliques does not intersect $S_1$.
    Since $S_2$ is an independent set in $G$, this clique must have exactly
    one vertex in $S_2$ and two vertices in $S_3$.  We can deduce, then, that
    $S_3$ contains exactly two vertices, and that these vertices are adjacent in $G$.

    Now consider graph $H'$.  Since $H'$ is an induced subgraph of $G$, it
    follows by taking complements of $H'$ and $G$ that $H$ is an induced
    subgraph of $G'$ (the complement of $G$).  We can repeat the argument
    of the previous paragraph with the roles of $S_1$ and $S_2$ reversed to
    show that the two vertices in $S_3$ must be adjacent in the complement of
    $G$, and therefore must not be adjacent in $G$.  Since we previously showed that
    these vertices are adjacent in $G$, we have a contradiction.
\end{proof}

\Cref{f6proposition} implies that $f(6) \geq 14$.

To find an induced universal graph for $\calF(6)$ in order to give an upper bound
on $f(6)$, we use the following local search
algorithm, which is shown in \Cref{alg:heuristic}.
The entry point to the algorithm is function $\FuncSty{FindInducedUniversalGraph}()$,
which takes as its parameter the maximum number of attempted improvement steps to
take before restarting from scratch.  The main search function, which is called
until timeout, is $\FuncSty{FindInducedUniversalGraph'}()$.

\begin{algorithm}[h!]
\DontPrintSemicolon
\nl $\FuncSty{FlipEdge}(G, v, w)$  \;
\nl \Begin{
\nl      \If {$G$ has edge $\{v,w\}$}{
    \nl           $\KwSty{return}$~$(V(G), E(G) \setminus \{v, w\})$ \;
}
\nl \Else {
    \nl           $\KwSty{return}$~$(V(G), E(G) \cup \{v, w\})$ \;
}
}
\medskip
\nl $\FuncSty{Score}(G)$  \;
\nl \Begin{
    \nl $\KwSty{return}$~$|\{H \in \calF(6) \mid H \textrm{ is an induced subgraph of } G\}|$ \;
}
\medskip
    \nl $\FuncSty{FindInducedUniversalGraph'}(\AlgVar{maxIter})$  \;
\nl \Begin{
\nl    $G \gets (\{0, \dots, 13\}, \emptyset)$ \LeftComment{A graph on 14 vertices with no edges}\;
\nl    Add edges to make vertices $\{0, \dots, 5\}$ of $G$ a clique \;
\nl    Add to $G$ each possible edge marked $\boldsymbol{\cdot}$ in \Cref{fig:heuristic-regions} with probability $1/2$ \;
    \nl    \For {$i \in \{1, \dots \AlgVar{maxIter}\}$\label{ImproveLoop}}{
\nl      $v, w \gets$ vertices corresponding to a randomly-selected position marked $\boldsymbol{\cdot}$ in \Cref{fig:heuristic-regions} \;
    \nl      $G' \gets \FuncSty{FlipEdge}(G, v, w)$ \label{FlipStep} \;
    \nl      \lIf {$\FuncSty{Score}(G') = |\calF(6)|$}{$\KwSty{return}$~$G'$ \LeftComment{Success!}}
    \nl      \lIf {$\FuncSty{Score}(G') \geq \FuncSty{Score}(G)$}{$G \gets G'$\label{ImproveLoopEnd}}
       }
\nl   $\KwSty{return}$~$\AlgVar{null}$ \;
    }
\medskip
    \nl $\FuncSty{FindInducedUniversalGraph}(\AlgVar{maxIter})$  \;
\nl \Begin{
\nl    \While {time limit is not exceeded}{
    \nl    $G \gets \FuncSty{FindInducedUniversalGraph'}(\AlgVar{maxIter})$ \;
    \nl    \lIf {$G \not= \AlgVar{null}$}{$\KwSty{return}$~$G$}
       }
\nl   $\KwSty{return}$~$\AlgVar{null}$ \;
    }
\caption{A hill-climbing algorithm to find an induced universal graph for family $\calF(6)$}
\label{alg:heuristic}
\end{algorithm}

The function $\FuncSty{FindInducedUniversalGraph'}()$
begins by generating a random graph $G$ on 14 vertices as follows
(see \Cref{fig:heuristic-regions}).
Numbering the vertices from zero, we make the
$k$ vertices numbered $0$ to $k-1$ a clique, and the $k$ vertices numbered $k-1$ to $2k-2$ an
independent set.\footnote{The
clique and the independent set thus have one vertex
in common.  The proof of Proposition~\ref{f6proposition} can be modified straightforwardly
to show that there is no 14-vertex induced universal graph for this family of graphs
that contains a 6-vertex clique and a 6-vertex independent set as
induced subgraphs with disjoint vertex sets.}

\begin{figure}[h!]
    \centering
    \footnotesize
       \begin{tikzpicture}[>=latex',line join=bevel,scale=.4]
        \tikzstyle{every node}=[]
          \input{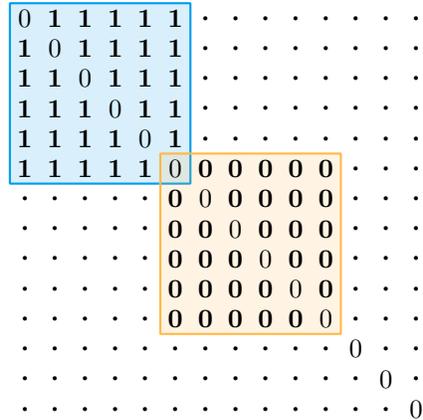}
      \end{tikzpicture}
    \caption{The adjacency matrix used in our heuristic algorithm when searching
        for an induced universal graph of order 14 for $\calF(6)$.  The blue
        region (a 6-vertex clique) and the yellow region (a 6-vertex independent
        set) remain unchanged as the algorithm runs.  The remaining possible edges,
        marked $\boldsymbol{\cdot}$, are initially set to 0 or 1 at random,
        and are flipped between 0 and 1 as the algorithm progresses.}
\label{fig:heuristic-regions}
\end{figure}

Each possible edge that is not involved in either the clique or the independent
set --- shown in the figure with a dot --- is added with probability $1/2$.
The loop in \linerangeref{ImproveLoop}{ImproveLoopEnd} then carries out up to
$\AlgVar{maxIter}$ iterations of a step that randomly modifies graph $G$,
and accepts the modification if it causes the number of graphs in $\calF(6)$ that
are induced subgraphs of $G$ to increase.
The attempted-improvement step chooses a random pair of vertices $\{v, w\}$,
both of which are not in the same coloured square
in \Cref{fig:heuristic-regions}, and
``flips'' the status of the edge between these two vertices: the edge is added if it
is not present, and removed otherwise.
After each flip, we count the number of graphs on 6 vertices (from a total
of 156 graphs) that are
isomorphic to an induced subgraph of our 14-vertex graph.  If the most recent flip
increased this number or left it unchanged, we accept the modification
to the graph.  After $\AlgVar{maxIter}$
flips, we restart the algorithm with a new randomly-generated graph.

Our implementation has two optimisations that are not shown in \Cref{alg:heuristic}.
First, scores are cached in a dictionary data structure, so that we can quickly look
up the score of a graph that has already been visited. This data structure is
cleared on return from $\FuncSty{FindInducedUniversalGraph'}$, in order to
bound the program's memory use.
(An alternative approach, which we have not implemented but which
may be useful in order to explore more of the space of possible graphs, would be to
use a tabu list \cite{DBLP:books/daglib/0093574} to avoid revisiting graphs
that have been visited recently.)

The second optimisation often avoids the need to visit every graph in $\calF(6)$
during calls to $\FuncSty{Score}()$.  Rather than beginning our count at 0
and incrementing it for every graph in $\calF(6)$ that is an induced subgraph
of $G$, we begin our count at $|\calF(6)|$ and decrement it for every
graph that is \emph{not} an induced subgraph of $G$.  We can then return
early from the function with a score of $-1$ as soon as the counter variable
falls below the score obtained by the graph prior to the most recent edge-flip.
This does not affect the behaviour of the hill-climbing search, since any
score below the previous graph's score results in the most recent change being rejected.

By running this search algorithm, we quickly found the graph of order 14 whose adjacency
matrix is shown in \Cref{fig:adjmat14}.  This is induced universal for $\calF(6)$;
therefore, $f(6) \leq 14$.  Combining our two bounds, we have $f(6) = 14$.

\begin{figure}[htb]
\centering
\scriptsize
\verb|0 1 1 1 1 1 0 1 1 0 0 0 1 0| \\
\verb|1 0 1 1 1 1 1 1 1 1 1 0 0 1| \\
\verb|1 1 0 1 1 1 0 1 1 0 1 0 1 0| \\
\verb|1 1 1 0 1 1 1 0 0 0 1 1 0 0| \\
\verb|1 1 1 1 0 1 0 1 0 0 0 0 0 0| \\
\verb|1 1 1 1 1 0 0 0 0 0 0 1 0 1| \\
\verb|0 1 0 1 0 0 0 0 0 0 0 1 1 1| \\
\verb|1 1 1 0 1 0 0 0 0 0 0 1 1 0| \\
\verb|1 1 1 0 0 0 0 0 0 0 0 0 1 0| \\
\verb|0 1 0 0 0 0 0 0 0 0 0 1 0 1| \\
\verb|0 1 1 1 0 0 0 0 0 0 0 1 1 0| \\
\verb|0 0 0 1 0 1 1 1 0 1 1 0 0 1| \\
\verb|1 0 1 0 0 0 1 1 1 0 1 0 0 0| \\
\verb|0 1 0 0 0 1 1 0 0 1 0 1 0 0|
\caption{The adjacency matrix of a 14-vertex induced universal graph for the family of all
six-vertex graphs}
\label{fig:adjmat14}
\end{figure}

\section{Bounds on \texorpdfstring{$f(7)$}{f(7)}}\label{sec:f7}

By Proposition~\ref{f6proposition}, we have $f(7) \geq 16$.  Figure~\ref{fig:adjmat18}
shows the adjacency matrix of an 18-vertex induced universal
graph for the family of all seven-vertex graphs. This was generated with 
the heuristic described in Section~\ref{sec:f6}, with two modifications.
First, 10,000 rather than 1000
edge-flips were permitted before each restart, as this was found to be more effective
in a preliminary run of the experiment.  Second, the overlapping six-vertex clique
and independent set were replaced with a clique on seven vertices and an independent
set on seven vertices.  Again, these had one vertex in common.  (We also tried making
the clique and independent set vertex-disjoint, but did not find an 18-vertex solution
in four hours with this approach.)

Thus we have $16 \leq f(7) \leq 18$.

\begin{figure}[htb]
\centering
\scriptsize
\verb|0 1 1 1 1 1 1 1 1 1 0 1 1 1 1 1 0 0| \\
\verb|1 0 1 1 1 1 1 0 1 1 1 0 1 1 0 1 0 0| \\
\verb|1 1 0 1 1 1 1 1 0 1 0 1 0 0 0 0 0 1| \\
\verb|1 1 1 0 1 1 1 0 1 0 0 1 0 0 0 0 0 1| \\
\verb|1 1 1 1 0 1 1 1 1 0 0 0 0 0 1 0 0 0| \\
\verb|1 1 1 1 1 0 1 0 1 0 0 0 0 1 1 1 1 1| \\
\verb|1 1 1 1 1 1 0 0 0 0 0 0 0 1 0 0 0 1| \\
\verb|1 0 1 0 1 0 0 0 0 0 0 0 0 0 1 1 1 0| \\
\verb|1 1 0 1 1 1 0 0 0 0 0 0 0 1 1 1 0 1| \\
\verb|1 1 1 0 0 0 0 0 0 0 0 0 0 0 1 1 1 1| \\
\verb|0 1 0 0 0 0 0 0 0 0 0 0 0 1 0 0 1 0| \\
\verb|1 0 1 1 0 0 0 0 0 0 0 0 0 0 1 1 1 0| \\
\verb|1 1 0 0 0 0 0 0 0 0 0 0 0 0 0 1 1 1| \\
\verb|1 1 0 0 0 1 1 0 1 0 1 0 0 0 0 0 0 1| \\
\verb|1 0 0 0 1 1 0 1 1 1 0 1 0 0 0 0 0 1| \\
\verb|1 1 0 0 0 1 0 1 1 1 0 1 1 0 0 0 1 1| \\
\verb|0 0 0 0 0 1 0 1 0 1 1 1 1 0 0 1 0 1| \\
\verb|0 0 1 1 0 1 1 0 1 1 0 0 1 1 1 1 1 0|
\caption{The adjacency matrix of an 18-vertex induced universal graph for the family of all
seven-vertex graphs}
\label{fig:adjmat18}
\end{figure}

\section{Trees}\label{sec:trees}

To generate induced universal graphs for families of
all trees on $k$ vertices, we tried two approaches.
The first approach was to run \Cref{alg:brute-force}, with the family
$\calF$ of order-$k$ trees loaded from
a list published by Brendan
McKay.\footnote{\url{https://users.cecs.anu.edu.au/~bdm/data/trees.html}}
The second approach, described in the following subsection, generates
only candidate graphs that contain a star --- one of the elements
of $\calF$ ---as an induced subgraph.  The final part of this section
gives results generated using these two methods, along with additional
upper bounds generated using a specialised version of our hill-climbing
algorithm.

\subsection{The completion method}

In our second approach to finding induced universal graphs for the
family of all trees on $k$ vertices, we embed the star $S_{k-1}$ --- which must be
present since it is a tree on $k$ vertices --- in the top left position
of the adjacency matrix, then systematically try all possible ways to complete
the adjacency matrix.  For each completed adjacency matrix, we test whether
the graph contains as an induced subgraph each tree in our family.
If it does, and if the graph is not isomorphic to any induced universal
graph that has already been found, we add it to our collection of graphs.
We will refer to this as the ``naive completion'' method.

Because the naive completion method has to test a huge number of graphs,
many pairs of which are isomorphic, we created an improved version of the
algorithm, which we will refer to as the ``symmetry breaking completion''
method.  Our description refers to the adjacence matrix in
\Cref{fig:regions-for-trees}.
This adds two symmetry-breaking constraints on the adjacency matrix.
The first of these constraints concerns the grey shaded region at the bottom-right
of the adjacency matrix, which has $n - k$ rows and columns.  Rather than
trying all $2^{n-k \choose 2}$ possible entries for this region of the
adjacency matrix, we use a list containing a canonical adjacency matrix
for each possible graph of order $n-k$, and we require that the shaded
region be equal to one of these adjacency matrices.

\begin{figure}[h!]
    \centering
    \footnotesize
       \begin{tikzpicture}[>=latex',line join=bevel,scale=.6]
        \tikzstyle{every node}=[]
          \input{regions-for-trees}
      \end{tikzpicture}
    \caption{The adjacency-matrix regions used in the 
        ``symmetry breaking completion'' method. For this example, $n=11$ and $k=7$.}
\label{fig:regions-for-trees}
\end{figure}

The second symmetry-breaking constraint concerns the yellow shaded region
in \Cref{fig:regions-for-trees}. Each row of this region has $n-k$ entries,
named $x_{i1}, \dots, x_{i,n-k}$.  For our symmetry-breaking constraint, we
view each row as the binary number $x_{i1}, \dots, x_{i,n-k}$
(with $x_{i,n-k}$ as the least significant digit), and we insist
that these numbers are in nondecreasing order; thus
$x_{21}, \dots, x_{2,n-k} \leq \dots \leq x_{k1}, \dots, x_{k,n-k}$.
In other words, we impose a lexicographic ordering constraint
on the rows of this region.

The algorithm with these symmetry-breaking constraints is shown
in \Cref{alg:completion}.  The call to $\FuncSty{MakeGraph}()$ on
\lineref{MakeGraphLine} constructs a graph by building its
adjacency matrix as shown in \Cref{fig:regions-for-trees}.  The
first $k$ vertices are connected in the form of a star;
for each $i$ such that $1 \leq i \leq j$, the digits of the binary
representation of $x_i$
are assigned to $x_{i1}, \dots, x_{i,n-k}$, and these values
are placed in the adjacency matrix as shown in the figure.  Finally,
the small adjacency matrix $M$ is copied to the bottom-right of
the adjacency matrix of the constructed graph.

For the collection of graphs $\calG$ defined on \lineref{DefineSetG},
isomorphic graphs are considered identical.  Thus, we only add a graph
to $\calG$ if it is not isomorphic to any existing member of the
collection.\footnote{An efficient method to ensure that isomorphic
graphs are not stored in $\calG$ would be label the vertices
of a graph $G$ canonically before adding it to the collection
\cite{DBLP:journals/jsc/McKayP14}, and to store the collection as a
hash set.  Because this part of the program is not a performance
bottleneck, and to miminise software dependencies, our implementation
uses the simpler approach of testing whether $G$ is isomorphic
to each existing member of $\calG$, in turn.}

\begin{algorithm}[h!]
\DontPrintSemicolon
\nl $\FuncSty{Search}(n,k,(x_1, \dots, x_j))$ \;
\nl \KwData{Natural numbers $n$ and $k$, and a sequence
    $(x_1, \dots, x_j))$ of nonnegative integers such that the binary
    representation of each $x_i$ represents entries $x_{i1}, \dots, x_{i,n-k}$
    of an adjacency matrix, as shown in \Cref{fig:regions-for-trees}}
\nl \KwResult{The set of all order-$n$ graphs that are induced universal
    for the family of trees of order $k$, and have an adjacency matrix
    given by $\FuncSty{MakeGraph}(n,k,X,M)$ for some completion
    $X$ of $(x_1, \dots, x_j)$ and some $M \in \mathcal{M}(n-k)$}
\nl \Begin{
\nl    $\calG \gets \emptyset$ \label{DefineSetG} \;
\nl    \If {$j \leq 1$}{
\nl      \For {$i \in \{0, \dots, 2^{n-k} - 1\}$}{
\nl        $\calG \gets \calG \cup \FuncSty{Search}(n,p,(x_1, \dots, x_j, i))$
  }
}
\nl    \ElseIf {$j \leq k$}{
\nl      \For {$i \in \{0, \dots, x_j\}$}{
\nl        $\calG \gets \calG \cup \FuncSty{Search}(n,p,(x_1, \dots, x_j, i))$
  }
}
\nl    \Else {
\nl      \For {$M \in \mathcal{M}(n-k)$}{
\nl        $G \gets \FuncSty{MakeGraph}(n, k, (x_1, \dots, x_j), M)$ \label{MakeGraphLine}\;
\nl        \If {$G$ contains every tree of order $k$ as an induced subgraph}{
\nl          $\calG \gets \calG \cup \{G\}$ \;
  }
}
}
\nl    $\KwSty{return}$~$\calG$ \;
    }
\medskip
\nl $\FuncSty{FindAllInducedUniversalGraphsForTrees}(n,k)$ \;
\nl \KwData{Natural numbers $n$ and $k$ such that $k<n$}
\nl \KwResult{The set of all graphs of order $n$ that are induced universal for
    the family of trees of order $k$.}
\nl \Begin{
\nl    $\KwSty{return}$~$\FuncSty{Search}(n,k,())$ \;
    }
\caption{The symmetry-breaking completion algorithm for finding all graphs of order $n$
    that are induced universal for the family of all order-$k$ trees}
\label{alg:completion}
\end{algorithm}

We claim that although the symmetry breaking completion method does not
visit every adjacency matrix visited by the naive completion method,
it finds, up to isomorphism, all induced universal graphs for the family
of trees on $k$ vertices.  This is shown by the following proposition.

\begin{proposition}
    Let $n, k$ be given, and let $G$ be a graph on $n$ vertices that is
    induced universal for the family of trees of order $k$.  Let $\mathcal{M}$
    be a set of canonical adjacency matrices for $\calF(n-k)$.  There is
    a graph on $n$ vertices numbered $1, \dots, n$ that is isomorphic
    to $G$ and satisfies the two symmetry-breaking constraints described
    above.
\end{proposition}
\begin{proof}
    Since $G$ is induced universal for the family of trees of order $k$, it
    contains the star $S_{k-1}$ as an induced subgraph.  Arbitrarily
    choose one such star in $G$ and give the label 1 to its vertex of degree
    $k-1$.  Its remaining vertices will be given the labels $\{2, \dots, k\}$,
    in an order that will be specified in the next paragraph.
    The remaining $n-k$ vertices of $G$ induce a subgraph isomorphic
    to one of the adjacency matrices in $\mathcal{M}$; call this matrix $M$.
    By choosing an appropriate relabelling of the $n-k$ vertices of $G$,
    the region of the adjacency matrix corresponding to these vertices
    can be made equal to $M$.  Thus, the first symmetry-breaking constraint
    is satisfied.

    The second symmetry-breaking constraint can now be satisfied by assigning
    labels $\{2, \dots, k\}$ to the degree-1 vertices of the star $S_{k-1}$
    in such an order that their adjacencies to the vertices outside the star,
    viewed as binary numbers, are in lexicographic order.
\end{proof}

\subsection{Results for families of trees}

The symmetry-breaking completion method is much faster than \Cref{alg:brute-force}.
For $n=9$ and $k=6$, for example, the completion method takes 9 seconds whereas
the brute-force algorithm takes 46 seconds. The difference in run-times becomes
even greater for larger values of $n$ and $k$.
Using the symmetry-breaking completion method,
we were able to find the values of $t(k)$ and $T(k)$ for $k \leq 7$. We were also
able to show that $t(8) > 12$.  We confirmed
the results for $k \leq 6$ using \Cref{alg:brute-force}.

Table~\ref{tab:treeresults} gives the order $t(k)$ of a minimal induced universal graph for
the family of $k$-vertex trees, and the number $T(k)$ of such graphs.  Figure~\ref{fig:tree-example}
shows one of the 687 minimal induced universal graphs for the family of 7-vertex trees.

\begin{table}[h!]
\centering
\begin{tabular}{r r r}
 \toprule
 $k$ & $t(k)$ & $T(k)$ \\ [0.5ex]
 \midrule
 1 & 1 & 1 \\
 2 & 2 & 1 \\
 3 & 3 & 1 \\
 4 & 5 & 2 \\
 5 & 7 & 18 \\
 6 & 9 & 66 \\
 7 & 11 & 687 \\
 \bottomrule
\end{tabular}
\caption{For each $k$, $t(k)$ is the minimum order of a graph containing all $k$-vertex trees as
induced subgraphs, and $T(k)$ is the number of distinct $t(k)$-vertex graphs that contain
all $k$-vertex trees as induced subgraphs.}
\label{tab:treeresults}
\end{table}

\begin{figure}[htb]
    \centering
    \input{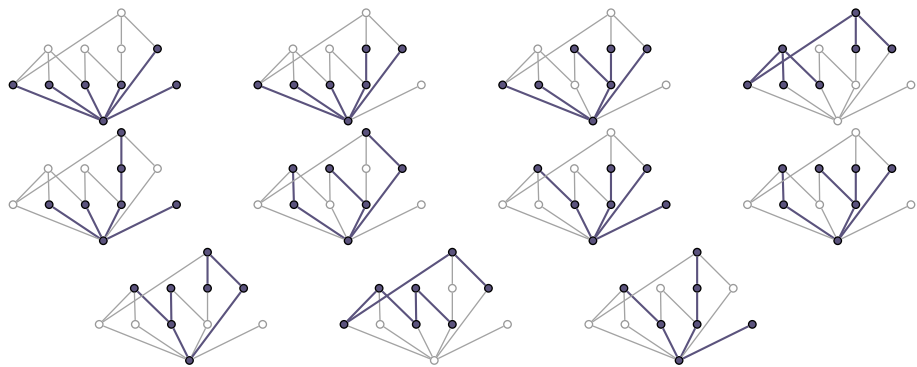}
\caption{An induced universal graph for the family of all
    trees on 7 vertices.  In each copy of the graph, one of the
    11 trees of order 7 is highlighted.}.
\label{fig:tree-example}
\end{figure}

To obtain upper bounds on $t(k)$ for larger values of $k$, we use a version of
\Cref{alg:heuristic} with minor modifications.  Instead of initially embedding
a clique and an independent set as in \Cref{fig:heuristic-regions}, we embed
a single tree --- the $k$-vertex star $S_{k-1}$ --- as shown in
\Cref{fig:heuristic-regions-for-trees}.  Each possible edge marked with a dot
in the figure is added with probability 0.1.

\begin{figure}[h!]
    \centering
    \footnotesize
       \begin{tikzpicture}[>=latex',line join=bevel,scale=.4]
        \tikzstyle{every node}=[]
          \input{heuristic-regions-for-trees}
      \end{tikzpicture}
    \caption{The adjacency matrix used in our heuristic algorithm when searching
        for an induced universal graph of order 9 for the set of all trees
        of order 6.  The blue
        region, a star with 6 vertices, remains unchanged as the algorithm runs.
        The remaining possible edges, marked $\boldsymbol{\cdot}$,
        are flipped between 0 and 1 as the algorithm progresses.}
\label{fig:heuristic-regions-for-trees}
\end{figure}

Using this heuristic algorithm, we were able to find induced universal graphs
to demonstrate that $t(8) \leq 13$, $t(9) \leq 15$ and $t(10) \leq 18$.
Since we have matching lower and upper bounds for $t(8)$, we can state
that $t(8) = 13$.

\section{Verification}\label{sec:verification}

To verify correctness of the results in
\Cref{tab:graphresults} and \Cref{tab:treeresults}, we repeated every call to
the McSplit subgraph isomorphism solver using the LAD algorithm \cite{DBLP:journals/ai/Solnon10}
as implemented in the \texttt{igraph} library \cite{igraph}.

We verified that the graphs shown as adjacency matrices in \Cref{fig:adjmat14}
and \Cref{fig:adjmat18} are induced universal graphs for their corresponding families
using a shell script that calls the Glasgow Subgraph Solver \cite{DBLP:conf/cp/McCreeshP15,DBLP:conf/gg/McCreeshP020} to check for the
inclusion of every graph of $k$ vertices.
This script shares no code with the Python program that generated the graphs.

Finally, we verified all but the final row in each of \Cref{tab:graphresults}
and \Cref{tab:treeresults} using a second shell script that calls the Glasgow
Subgraph Solver.  Again, this shares no code with the program used to generate
the results in the tables.

\section{Acknowledgements}

I would like to thank Brendan McKay for helpful feedback on this paper, and Persi
Diaconis for introducing me to induced universal graphs and for interesting email
discussions on related topics.

\bibliographystyle{abbrv}
\bibliography{main}
\end{document}